\documentclass[12pt,a4paper]{amsart}
\usepackage[cp1250]{inputenc}

\usepackage[OT4]{fontenc} 

\newtheorem{tw}{Theorem}[section]
\newtheorem*{re}{Remark}
\newtheorem{df}[tw]{Definition}
\newtheorem{co}[tw]{Corollary}

\newtheorem{lem}[tw]{Lemma}

\theoremstyle{definition}
\newtheorem{ex}[tw]{Example}

\author{Patryk Pagacz}
\title{On the power-bounded operators of classes $C_{0 \cdot}$ and $C_{1 \cdot}$}

\newcommand{\h}{\ensuremath{\mathcal H}}
\newcommand{\N}{\ensuremath{\mathbb N}}
\newcommand{\M}{\ensuremath{\mathcal M}}
\newcommand{\K}{\ensuremath{\mathcal K}}

\usepackage{amssymb}

\begin{document}
\maketitle

\begin{abstract}
By a bounded backward sequence of the operator $T$ we mean a bounded sequence $\{x_n\}$ satisfying $Tx_{n+1}=x_n$.
In \cite{Pa} we have characterized contractions with strongly stable nonunitary part in terms of bounded backward sequences.

The main purpose of this work is to extend that result to power-bounded operators.

Aditionally, we show that a power-bounded operator is strongly stable ($C_{0 \cdot} $) if and only if its adjoint does not have any nonzero bounded backward sequence.
Similarly, a power-bounded operator is non-vanishing ($C_{1 \cdot} $) if and only if its adjoint has a lot of bounded backward sequences.
\end{abstract}

\footnote{
AMS(MOS) subject classifcation (2010): 47A05, 47A45, 47B37, 47G10.
}

\section{Preliminaries} Let $\h$ be a complex, separable Hilbert space. We denote by $\mathcal{B}(\h)$ the space of bounded linear transformations acting on $\h$. By a \emph{contraction} we mean
$T\in \mathcal{B}(\h)$ such that $\|Tx\|\leq \|x\|$ for each $x\in \h$. By a \emph{power-bounded} operator we mean $T\in \mathcal{B}(\h)$ such that $\|T^n\|$ is uniformly bounded for all $n=1,2,3,...$.

An operator $T$ is said to be \emph{completely nonunitary} (abbreviated \emph{cnu}) if $T$ restricted to every reducing subspace of \h\ is nonunitary. As usual, by $T^*$ we mean the adjoint of $T$.

We define as usually:
\begin{df} An operator $T\in \mathcal{B}(\h)$ is said to be of class $C_{0 \cdot}$ if $$\liminf\limits_{n\to \infty} \|T^n x\|= 0$$ for each $x \in \h$. \end{df}

Note that a power-bounded operator $T$ is of class $C_{0 \cdot}$ if and only if it is strongly stable ($T^n \to 0$, SOT).

Indeed, let $T$ be $C_{0 \cdot}$. If we fix $x \in \h$ then for each $\epsilon >0$ there is $k\in \N$ such that $\|T^kx\|<\epsilon$, so for all $m>k$ we have $$\|T^mx\| = \|T^{m-k}T^kx\|\leq
\|T^{m-k}\|\|T^kx\| \leq  \epsilon\sup\limits_{n\in \N} \|T^n\|.$$

In general, $C_{0 \cdot}$ operators can be extremely different from strongly stable operators. The following example shows a bounded operator of class $C_{0 \cdot}$, which is not strongly stable at any (nonzero) point.

\begin{ex}
Let $\{N_k\}_k$ be the sequence such that
$$\left\{
\begin{array}{rl}
 N_1&=1 \\
 N_{k+1}&=3N_k+2N_k^2 \textnormal{ , for } k=1,2,... \\
\end{array}
\right. $$

Then let us define the operator $S$ as the unilateral shift with weights $w_1,w_2,...$, i.e.,  $S : l^2 \ni (x_1,x_2,...) \mapsto (0,w_1x_1,w_2x_2,...) \in l^2$,
where
$$\left\{
\begin{array}{rl}
 w_1&=1 \\
 w_i&=\frac12 \textnormal{ , for } i=N_k+1,N_k+2,...,3N_k \\
 w_i&=2^{\frac{1}{N_{k}}} \textnormal{ , for } i=3N_k+1,3N_k+2,...,N_{k+1}.
\end{array}
\right. $$

For nonzero $x=(x_1,x_2,...)\in l^2$ there is $i_0$ such that $x_{i_0} \not= 0$.\\
But by definition of $S$ we have $S^{N_k}e_1=e_{N_k+1}$ for all $k\in\N$, thus $\|S^{N_k-i_0}x\|\geq \|S^{N_k-i_0}x_{i_0}e_{i_0}\|=\frac{1}{w_1w_2\cdot...\cdot w_{i_0}}|x_{i_0}|$.\\So $S^nx\not\to 0$ for all nonzero $x \in l^2$.

Now we show that $S$ is of class $C_{0 \cdot}$.\\
Fix $x=(x_1,x_2,...)\in l^2$ and $\epsilon>0$. We can assume $\|x\|=1$.\\
Since $\{N_k\}_k$ increases, then there is $N\in\{N_k| k=1,2,...\}$ such that $\sum\limits_{i=N+1}^{\infty} |x_i|^2 < \frac{\epsilon}{32}$ and $\left(\frac{1}{2^N}\right)^2 < \frac{\epsilon}{2}$. By that we obtain:
$$\|S^{2N}x\|^2 =\sum\limits_{i=1}^{N} |x_i|^2\|S^{2N}e_i\|^2 + \sum\limits_{j=N+1}^{\infty} |x_j|^2\|S^{2N}e_j\|^2 =$$
$$=\sum\limits_{i=1}^{N}|x_i|^2\|S^N(w_{i}w_{i+1}...w_{i+N-1}e_{i+N})\|^2+\sum\limits_{j=N+1}^{\infty}|x_j|^2|w_{j}w_{j+1}...w_{j+2N-1}|^2\leq$$
$$\leq\sum\limits_{i=1}^{N}|x_i|^2\|w_{i}w_{i+1}\cdot...\cdot w_N \cdot \left(\frac12\right)^{i-1} S^Ne_{i+N} \|^2+$$ $$+\sum\limits_{j=N+1}^{\infty}|x_j|^2| \underbrace{2^{\frac1N}2^{\frac1N}\cdot...\cdot2^{\frac1N}}_{2N} |^2
\leq \sum\limits_{i=1}^{N}|x_i|^2\|S^Ne_{i+N}\|^2+\sum\limits_{j=N+1}^{\infty}|x_j|^2 4^2=$$ $$=\sum\limits_{i=1}^{N}|x_i|^2\|\left(\frac12\right)^Ne_{i+2N}\|^2 +\frac{\epsilon}{32}16 \leq \|x\|\frac{\epsilon}{2} + \frac{\epsilon}{2} = \epsilon. \ \ \square$$
\end{ex}

In contrast to the above notion we have:

\begin{df} An operator $T\in \mathcal{B}(\h)$ is said to be of class $C_{1 \cdot}$ if $$\liminf\limits_{n\to \infty} \|T^n x\|> 0$$ for each nonzero $x \in \h$. \end{df}

Operators of class $C_{1 \cdot}$ are also called \emph{non-vanishing}.\\
We also say that $T$ is of class $C_{\cdot 0}$ or $C_{\cdot 1}$ if its adjoint is of class $C_{0 \cdot}$ or $C_{1 \cdot}$, respectively.

Let us define $\mathcal{M}(T) := \{x \in \mathcal{H} |\  \exists \{x_n\}_{n \in \mathbb{N}} : x=x_0, Tx_{n+1}=x_n $ and $ \{x_n\}_{n \in \mathbb{N}} $ is bounded $ \}$.
Naturally, such a sequence $\{x_n\}_{n\in \N}$ can be called as the \emph{bounded backward sequence}.
\section{Introduction}

In the paper \cite{Pa} we have presented the following theorem with some applications.
\begin{tw}\label{DK} Let $T$ be a contraction. The following conditions are equivalent:
\begin{itemize}
\item for any bounded backward sequence $\{x_n\}_{n\in \N}$ of $T$, the sequence of norms $\{\|x_n\|\}_{n\in \N}$ is constant,
\item the nonunitary part of $T$ is of class $C_{\cdot 0}$.
\end{itemize}
\end{tw}

We have been asked the natural question about a possible generalization to power-bounded operators.
In this work we will try to answer this question.

The easy extension of the above theorem for power-bounded operators is not true.
To see this, let us consider the following: \begin{ex}

$$\textnormal{Let } T: l^2 \ni (x_1,x_2,x_3,...) \mapsto (0,x_1+x_2,0,x_3+x_4,0,...)\in l^2.$$ It is clear that $T$ is power-bounded, in fact $T=T^2$. Additionally, if $x=(a_1,a_2,a_3,...)\in \M(T) \subset T(\h)$, then $a_{2k+1}=0$ for all $k\in \N$.\\ Hence $T^{-1}(\{x\})=\{x\}$. Thus, even any (not necessary bounded) backward sequence of $T$ must be constant.

On the other hand, $T$ has trivial unitary part and is not $C_{\cdot 0}$,\\
since $T^*=T^{*2}$.
\end{ex}

\section{Characterization of $C_{1 \cdot}$ and $C_{0 \cdot}$ power-bounded operators}

To introduce the next theorem, let us recall the construction of isometric asymptotes (see \cite{Ke}).

Let us define a new semi-inner product on $\h$: $$[x,y] := \textnormal{glim} \{\langle T^{*n}x, T^{*n}y \rangle\}_{n\in \mathbb{N}},$$ where \textnormal{glim} denote a Banach limit.\\
Thus, the factor space $\h/\h_0$, where $\h_0$ stands for the linear manifold $\h_0:=\{x\in \h | [x,x]=0\}$, endowed with the inner product $[x+\h_0,y+\h_0] = [x,y]$, is an inner product space. Let $\K$ denote the resulting Hilbert space obtained by completion. Let $X$ denote the natural embedding of Hilbert space $\h$ into $\K$ i.e. $X : \h\ni x \mapsto x+\h_0 \in \K$.\\
We can see that: $\|XT^*x\|=\|Xx\|$.
So there is an isometry $V:\K\to \K$ such that $XT^*=VX$. The isometry $V$ is called \textit{isometric asymptote}.

\begin{lem}\label{lemma}

For any power-bounded operator $T \in \mathcal{B}(\h)$ the corresponding $X$ from the construction above satisfies:
$$X^*(\K)=\M(T).$$

\end{lem}
\begin{proof}
By definition of $V$ and $X$, we have $TX^*=X^*V^*$.
Let $x_n:=X^*V^nx$, then $$Tx_{n+1}= TX^*V^{n+1}x=X^*V^*V^{n+1}x=x_n.$$
Moreover $\|x_n\|\leq \|X^*\|\|x\|$, for all $n\in \mathbb{N}$.
Thus $X^*x=x_0\in \M(T)$, where $x \in \h$. Hence $X^*(\h)\subset\M(T)$.

To prove the converse, let us fix $x\in\M(T)$. By definition of $\M(T)$, there exists $\{x_n\}_{n\in \N}\subset \h$, a bounded backward sequence of $x$.\\
Let $y\in \h$, then
\begin{equation}|\langle x,y\rangle |= |\langle T^nx_n,y \rangle |= |\langle x_n,T^{*n}y \rangle |\leq \|x_n\|\|T^{*n}y\|.
\end{equation}
So $|\langle x,y\rangle |\leq \sup\limits_{n\in \mathbb{N}} \|x_n\| \liminf\limits_{n\to \infty} \|T^{*n}y\|\leq  \sup\limits_{n\in \mathbb{N}} \|x_n\| \|Xy\|$.
So by Theorem 1 in \cite{Sebe}, we have $x\in X^*(\K)$.
\end{proof}

At the begining, we have observed that if $\liminf\limits_{n\to \infty} \|T^{*n}x\|=0$, for some $x\in \h$, then $\lim\limits_{n\to \infty} \|T^{*n}x\|=0$. So we have:
 $$\{x\in \h| T^{*n}x\to 0\}=\mathcal{N}(X).$$

Now by Lemma \ref{lemma} we obtain:

\begin{co}\label{rozklad}
Let $T$ be a power-bounded operator, then
$$ \h = \{x\in \h| T^{*n}x\to 0\} \oplus \overline{\M(T)}. $$
\end{co}

We also have:

\begin{tw}\label{wn1} A power-bounded operator $T$ is $C_{\cdot 1}$ if and only if $\overline{\M(T)}=\h$. \end{tw}

It is trivial that $\M(T)$ is included in the set of origins of all backward sequences, that is, $T^{\infty}(\h) := \bigcap\limits_{n\in \N}T^n(\h)$. But in general, the converse inclusion does not hold, even for $C_{\cdot 1}$ contractions.

To see this let us consider the following example.\\
\begin{ex}
Let $\h = l^2$. Then $\mathbb{H} := l^2(\h) = \{ \{x_n\}_{n\in \N}\subset \h | \sum\limits_{n\in \N} \|x_n\|^2 < \infty \}$ is a
separable Hilbert space, with the norm $\|\{x_n\}_{n\in \N}\| := \sqrt{\sum\limits_{n\in \N} \|x_n\|^2}$. For the element $\{\textbf{x}_n\}_{n\in \N} \in \mathbb{H}$ sometimes we write $\bigoplus\limits_{n\in \N} \textbf{x}_n$.\\
Let $S_w$ be the backward unilateral shift with weights $w=(w_1,w_2,...)$, i.e., $S_w : \h \ni (x_1,x_2,...) \mapsto (w_1x_2,w_2x_3...) \in \h$. If we put $w_i^n =(\frac{1}{n})^{\frac{1}{i-1}-\frac{1}{i}}$ for all $n \in \N $ and $ i >2$, and $w_1^n=w_2^n=1$ for all $ n \in \N $, then $T=\bigoplus\limits_{n \in \N} S_{w^n}$ is a $C_{\cdot 1}$ contraction.\\
Indeed, $T^m=\bigoplus\limits_{n \in \N} S^m_{w^n}$, where $S^m_{w^n}(x_1,x_2,x_3,...) = (w_1^nw_2^n\cdot...\cdot w_m^nx_{m+1},w_2^nw_3^n\cdot...\cdot w_{m+1}^nx_{m+2},...)$ and $\lim\limits_{m\to\infty} w^n_1w^n_2\cdot...\cdot w^n_m =  \lim\limits_{m\to \infty} (\frac{1}{n})^{\frac{1}{2}-\frac{1}{m}} =\frac{1}{n^2} > 0 $.\\
Now, let us consider $x=\bigoplus\limits_{n \in \N} (\frac{1}{n},0,0,...)\in
\mathbb{H}$.\\
For $m\in \N$ we have $x=T^ma_m$, where \\
$a_m=\bigoplus\limits_{n \in \N} (\underbrace{0,0,...,0}_m,(\frac{1}{n})^{\frac{1}{2}+\frac{1}{m}},0,0,...) \in \mathbb{H}$. Thus
$x\in T^{\infty}(\h)$.\\ Now let $\{b_m\}_{m\in \N} \subset \mathbb{H}$ be a backward sequence for $x$, then \\
$x=T^mb_m$ and thus $b_m=\bigoplus\limits_{n \in \N}(q_1,q_2,...,q_m,(\frac{1}{n})^{\frac{1}{2}+\frac{1}{m}},0,0,...)$ for some complex $q_1,q_2,...,q_m$. So $\|a_m\|\leq \|b_m\|$, but \\
$\|a_m\|^2=\sum\limits_{n\in \N}(\frac{1}{n})^{2(\frac{1}{2}+\frac{1}{m})} \to \infty ($ for $ m\to \infty)$. Hence $x\not\in \M(T)$. \end{ex}

One more consequence of Corollary \ref{rozklad} is the following:

\begin{tw}\label{wn0} A power-bounded operator $T$ is $C_{\cdot 0}$ if and only if $\M(T) = \{0\}$. \end{tw}

Another proof of this theorem (in the case of operators considered on Banach spaces) can be found in \cite{Vu}.

\begin{co}\label{T^-1} If $T$ is power-bounded and invertible, then\\
$\|T^{*n}x\| \to 0$ for all $x\in \h$ if and only if $\|T^{-n}x\| \to \infty$ for all
$x\in\h$.
\end{co}
\begin{proof}
If $T$ is power-bounded, then  $T^*$ is power-bounded too.\\
By Theorem \ref{wn0} we obtain that $T^*$ is strongly stable if and only if each nontrivial sequence such that $Tx_{n+1}=x_n$ is unbounded.\\
But we have $x_n=T^{-n}x_0$. Thus the second condition means that $\sup\limits_{n\in \N}\|T^{-n}x\|=\infty$ for each nonzero $x\in \h$.\\
Now, if for some $x\in \h$ there is an increasing sequence $\{n_k\}_{k\in \N}$ such that $\sup\limits_{k\in \N}\|T^{-n_k}x\|<N$, then for each $n\in \N$ we have
$$\|T^{-n}x\|=\|T^{n_k-n}T^{-n_k}x\| \leq \|T^{n_k-n}\|\|T^{-n_k}x\|\leq N\sup\limits_{n\in \N}\|T^n\|, $$ since $n_k>n$ for some $k\in\N$.
\end{proof}

\begin{ex} Let $V$ be the classical integral Volterra operator defined, on the space $L^2[0,1]$, by \begin{center} $(Vf)(x):=\int_{0}^{x}f(t)dt$, for $f\in L^2[0,1]$. \end{center} It is easy to calculate that $(V^*f)(x)=\int_{x}^{1}f(t)dt$.\\
Hence $V+V^*=P$, where $P$ is the one-dimensional projection on subspace of constant functions. It is well-known that $\|(I+V)^{-1}\|=1$ (see Problem 150 in \cite{Ha}). The Allan-Pedersen relation (see \cite{Al}) $$S^{-1}(I-V)S=(I+V)^{-1},$$ where $Sf(t)=e^{t}f(t)$ show us that $I-V$ is similar to a contraction.\\
So it is power-bounded.  \\
Furthermore, Proposition 3.3 from
\cite{Le} yields to 
\begin{equation}\lim\limits_{n \to \infty} \sqrt{n}(I-V)^nVf = 0 \textnormal{, for all } f\in L^2[0,1]. \end{equation}
But as we mentioned, $I-V$ is power-bounded.
Moreover, $V$ has dense range. Therefore $I-V$ is $C_{0 \cdot}$. (To obtain this, instead of (2) we can use the Esterle-Katznelson-Tzafriri theorem (see \cite{Es}, \cite{KT}), since $\sigma(I-V)=\{1\}$.)\\
Now, by Corollary \ref{T^-1} we obtain
\begin{center}$ \|(I+V-P)^{-n}f\|=\|((I-V)^*)^{-n}f\| \to \infty $ for all $f \in L^2[0,1]\backslash\{0\}$. \end{center}
Additionally, form (1) in the proof of Lemma \ref{lemma} and (2) we have
\begin{center}$ \frac{1}{\sqrt{n}}\|(I+V-P)^{-n}f\| \to \infty $ for all $f \in V(L^2[0,1])\backslash\{0\}$. \end{center}
\end{ex}

\begin{re}
To obtain the first part of this result we also can use Theorem 3.4 form \cite{AD} and observe that each local spectrum $\sigma_x(I+V-P)$ is equal $\{1\}$. (Because $\{1\}=\sigma(I-V)=\sigma((I+V-P)^*)=\sigma(I+V-P)$.)
\end{re}

\begin{ex}
According to the above example we see that the contraction $(I+V)^{-1}$ is of class $C_{0 \cdot}$ and as before $\sigma(I+V)=\{1\}$. \\
So using Theorem 3.4 form \cite{AD} we obtain that $\|(I+V)^nf\|\to \infty$ for all nonzero $f \in L^2[0,1]$. \\
Now by Corollary \ref{T^-1} we have
\begin{center}$(I-V+P)^{-n}f=(I+V)^{-*n}f\to 0$ for all $f \in L^2[0,1]$.\end{center}
So the contraction $(I+V)^{-1}$ is of class $C_{0 0}$.
\end{ex}

\section{Main result}
To give a generalization of Theorem \ref{DK}, we will need the following lemma(due to K$\acute{\textnormal{e}}$rchy, see \cite{Ke}):

\begin{lem}\label{matrix}
If $T$ is power-bounded, then $T$ can be represented by the matrix \begin{equation} \begin{bmatrix}
T_{11} & T_{21} \\
0 & T_{22}
\end{bmatrix}, \end{equation} where $T_{11},T_{22}$ are power-bounded, $T_{11}$ is of class $C_{0 \cdot}$ and $T_{22}$ is of class $C_{1 \cdot}$.
\end{lem}
\begin{proof}
Let (3) be the matrix of $T$ with respect to the orthogonal decomposition $\h=\mathcal{N}\oplus \mathcal{N}^{\bot}$,
where $\mathcal{N}:=\{x\in \h |\  T^nx\to 0 \}$. By definition $\mathcal{N}$ is invariant for $T$. So $T\mid_{\mathcal{N}}=T_{11}$, thus $T_{11}$ is of class $C_{0 \cdot}$ (and power-bounded).
Moreover, we have:
$$T_{22}= P_{\mathcal{K}}T \in \mathcal{B}(\mathcal{K}) \textnormal{, where } \mathcal{K}:=\mathcal{N}^{\perp}\not=\{0\}.$$
The subspace $\K$ is invariant for $T^*$. So we obtain $T^*\mid_{\K} = T^*_{22} $, thus $T_{22}$ is power-bounded.

Now, we will show that $T_{22}$ is $C_{1 \cdot}$.

To see this, let us assume that $T_{22}^nf\to 0$ for some $f\in \mathcal{K}$.\\
For an arbitrary $\epsilon>0$, there is $n_0\in \N$ such that $\|T_{22}^{n_0}f\|<\frac{\epsilon}{2M}$, where $M:=\sup\limits_{n\in \N}\|T^n\|$.\\
Let us suppose for a while that $T_{22}\in \mathcal{B}(\mathcal{K},\h)$.
By definition of $T_{22}$ we have: $(T-T_{22})x=P_{\mathcal{N}}Tx \in
\mathcal{N}$ for each $x\in
\mathcal{K}$.\\
Hence for each $ k\in \{1,2,3,...,n_0\}$ there exists $m_k\in \N$ such that $\|T^{m'+k-1}(T-T_{22})T_{22}^{n_0-k}f\|\leq \frac{\epsilon}{2n_0}$ for all $m'\geq m_k$.\\
Now, for $m:=\max\{m_k |\  k=1,2,...,n_0\}$ we have:
\begin{center}
$\|T^{m+n_0}f\|=\|T^m(T^{n_0}-T^{n_0-1}T_{22}+T^{n_0-1}T_{22}-T^{n_0-2}T_{22}^2+...+\newline
+ TT_{22}^{n_0-1}-T_{22}^{n_0}+T_{22}^{n_0})f\|=
\|\sum\limits_{k=1}^{n_0}T^{m+k-1}(T-T_{22})T_{22}^{n_0-k}f+T^mT_{22}^{n_0}f\|
\leq
\newline \leq
\sum\limits_{k=1}^{n_0}\|T^{m+k-1}(T-T_{22})T_{22}^{n_0-k}f\|+\|T^mT_{22}^{n_0}f\|
\leq n_0\frac{\epsilon}{2n_0} + M\frac{\epsilon}{2M} = \epsilon.$
\end{center}
Thus $T^nf \to 0$, contrary to $f\in \mathcal{K}$.\\
So $T_{22}$ is of class $C_{1\cdot}$.
\end{proof}

Now, we can give our generalization of Theorem \ref{DK}:

\begin{tw}
Let $T$ be a power-bounded operator. The following conditions are equivalent:
\begin{itemize}
\item for any bounded backward sequence $\{x_n\}_{n\in \N}$ of $T$, the sequence of norms $\{\|x_n\|\}_{n\in \N}$ is constant;
\item $T$ can be decomposed as $T= \begin{bmatrix}
T_{11} & 0 \\
T_{21} & U
\end{bmatrix},$ where $U$ is a unitary and $T_{11}$ is of class $C_{\cdot 0}$.
\end{itemize}
\end{tw}
\begin{proof}
To the proof of the first implication, let $T= \begin{bmatrix}
T_{11} & 0 \\
T_{21} & T_{22}
\end{bmatrix}$ be the matrix form Lemma \ref{matrix}, where $T_{11}\in\mathcal{B}(\h_1)$ is $C_{\cdot 0}$ and $T_{22}\in\mathcal{B}(\h_2)$ is $C_{\cdot 1}$.
Now, $\h_2$ is invariant for $T$, thus $T_{22}=T|_{\h_2}$.
Hence, each bounded backward sequence of $T_{22}$ is bounded backward sequence of $T$. So, by our assumption $T_{22}$ is an isometry on $\overline{\M(T_{22})}$. But by Theorem \ref{wn1} we have $\overline{\M(T_{22})}=\h_2$.
So $T_{22}$ is an isometry.

Finally, it can be decomposed as $T_{22}=U \oplus S_+$, where $U$ is unitary and $S_+$ is the unilateral shift. But $T_{22}$ is $C_{\cdot 1}$. So we have $T_{22}=U$.

To prove the converse implication, let us assume that $\{x_n\}_{n\in \N}$ is the bounded backward sequence of $T$. Let $x_n=a_n+b_n$, where $a_n\in \h_1$ and $b_n\in \h_2$. We have: $$T_{11}a_{n+1}+(T_{21}a_{n+1}+Ub_{n+1})=Ta_{n+1}+Tb_{n+1}=Tx_{n+1}=x_n=a_n+b_n.$$
So $T_{11}a_{n+1}=a_n$ and $\|a_n\|\leq \|x_n\|$. It means that $\{a_n\}_{n\in \N}$ is a bounded backward sequence of $T_{11}$, but $T_{11}$ is of class $C_{\cdot 0}$. So by Theorem \ref{wn0} we obtain $a_n\equiv 0$. Thus $\|x_{n+1}\|=\|b_{n+1}\|=\|Ub_{n+1}\|=\|b_n\|=\|x_n\|$.
\end{proof}

\section{Acknowledgements}
I am very grateful to professor J. Zem$\acute{\textnormal{a}}$nek for his hospitality and discussions during my stay in the Institute of Mathematics of the Polish Academy of Sciences.

\bigskip
Instytut Matematyki, Uniwersytet Jagiello\'{n}ski,\\
\L ojasiewicza 6, 30-348, Krak\'{o}w, Poland\\
\emph{E-mail address: patryk.pagacz@im.uj.edu.pl}

\end{document}